\documentclass[11pt, reqno, english]{amsart}  
\usepackage[utf8]{inputenc}
\usepackage[T1]{fontenc}
\usepackage{amsmath,amsthm}
\usepackage{amsfonts,amssymb}
\usepackage{url}
\usepackage{mathtools}  
\usepackage[colorlinks=true,urlcolor=blue,linkcolor=red,citecolor=magenta]{hyperref}
\usepackage{enumerate, paralist}
\usepackage{tikz-cd}
\usepackage[capitalize]{cleveref}
\usepackage{thmtools}
\usepackage[left=1in,right=1in,top=1in,bottom=1in]{geometry}
\usepackage{microtype}

\theoremstyle{plain}
\newtheorem{theorem}{Theorem}[section]
\newtheorem{lemma}[theorem]{Lemma}%[section]
\newtheorem{corollary}[theorem]{Corollary}
\newtheorem{proposition}[theorem]{Proposition}

\theoremstyle{definition}
\newtheorem{definition}[theorem]{Definition}

\newtheorem{remark}[theorem]{Remark}

\newtheorem{question}[theorem]{Question}

\title[Simplicial complexes with $\pi_1 \cong \mathbb{Z}^n$]{Vertex numbers of simplicial complexes \\ with free abelian fundamental group}

\author{Florian Frick}
\address[FF]{Dept.\ Math.\ Sciences, Carnegie Mellon University, Pittsburgh, PA 15213, USA \newline \indent Inst. Math., Freie Universit\"at Berlin, Arnimallee 2, 14195 Berlin, Germany}
\email{frick@cmu.edu} 

\author{Matt Superdock}
\address[MS]{Dept.\ Mathematics and Computer Science, Rhodes College, Memphis, TN 38112, USA}
\email{superdockm@rhodes.edu}

\date{September 24, 2021}

\thanks{FF was supported by NSF grant DMS 1855591, NSF CAREER grant DMS 2042428, and a Sloan Research Fellowship. MS was supported by NSF grant DMS 1855591}

\begin{document}

\maketitle

\begin{abstract}
  \small
  We show that the minimum number of vertices of a simplicial complex with
  fundamental group $\mathbb{Z}^{n}$ is at most $O(n)$ and at least
  $\Omega(n^{3/4})$. For the upper bound, we use a result on orthogonal
  1-factorizations of $K_{2n}$. For the lower bound, we use a fractional
  Sylvester--Gallai result. We also prove that any group presentation $\langle S | R\rangle \cong
  \mathbb{Z}^{n}$ whose relations are of the form $g^{a}h^{b}i^{c}$ for $g, h, i
  \in S$ has at least $\Omega(n^{3/2})$ generators.
\end{abstract}

% ## Introduction
\section{Introduction}

Given a space~$X$, a \emph{vertex-minimal triangulation} of $X$ is a simplicial
complex homeomorphic to $X$ using as few vertices as possible. Such
triangulations are known for only a few manifolds~\cite{brehm1987combinatorial,
kuhnel1986higherdimensional, lutz2005triangulated}, and upper and lower bounds
differ significantly for many others, despite recent improvements such
as~\cite{adiprasito2020subexponential}. For example, the $n$-dimensional torus
can be triangulated on $2^{n + 1} - 1$ vertices~\cite{kuhnel1988combinatorial},
but the best known lower bounds are quadratic in~$n$;
see~\cite{arnoux1991kuhnel}. 

The number of faces of a simplicial complex $X$ can be bounded in terms of the
Betti numbers of $X$~\cite{billera1997face} or in terms of the minimal number of
generators of $\pi_{1}(X)$~\cite{murai2017face}. The effect of relations of
$\pi_{1}(X)$ on vertex numbers has been studied for cyclic torsion
groups~\cite{kalai1983enumeration, newman2019small} and for triangulations of
manifolds with non-free fundamental group~\cite{pavesic2019triangulations}. In
this paper, we consider the minimal number of vertices of a simplicial complex
with fundamental group~$\mathbb{Z}^n$:

\begin{theorem}
\label{asymptotic}
  We have the following asymptotic results for vertex numbers of simplicial
  complexes with fundamental group~$\mathbb{Z}^n$:

  \begin{enumerate}[(a)]
    \item
    There is a simplicial complex $X_{n}$ with $\pi_{1}(X_{n}) \cong
    \mathbb{Z}^{n}$ on $O(n)$ vertices.

    \item
    \label{asymptotic-lower}
    Every simplicial complex $X_{n}$ with $\pi_1(X_{n}) \cong \mathbb{Z}^{n}$
    has $\Omega(n^{3/4})$ vertices.
  \end{enumerate}
\end{theorem}

These results appear separately as \cref{upper,lower}; our precise upper bound
depends on parity, but is asymptotically~$4n$ in both cases:

\begin{restatable*}{theorem}{upper}
\label{upper}
  For $n \in \mathbb{N}$ with $n \ne 2, 3$, there exists:
  \begin{itemize}
    \item
    A simplicial complex $X_{2n}$ with $8n - 1$ vertices, $\pi_{1}(X_{2n}) \cong
    \mathbb{Z}^{2n}$.

    \item
    A simplicial complex $X_{2n - 1}$ with $8n - 3$ vertices, $\pi_{1}(X_{2n -
    1}) \cong \mathbb{Z}^{2n - 1}$.

  \end{itemize}
\end{restatable*}

To prove the $O(n)$ upper bound, we construct a complex $W_{n}$ on $n^{2} + n +
1$ vertices with fundamental group $\pi_{1}(W_{n}) \cong \mathbb{Z}^{n}$, and
then perform identifications that preserve $\pi_{1}(W_{n})$. The latter step
uses a result on orthogonal 1-factorizations of the complete graph $K_{2n}$,
which is implied by a result on Room squares; see
\cite{mullin1975existence,horton1981room,mendelsohn1985one}.

To prove the $\Omega(n^{3 / 4})$ lower bound, we relate simplicial complexes to
group presentations. Specifically, we define a \emph{3-presentation} as a group
presentation $\langle S | R\rangle$ whose relations are of the form
$g^{a}h^{b}i^{c}$ for $g, h, i \in S$; this is a generalization of triangular
presentations as studied
in~\cite{antoniuk2014collapse,antoniuk2015random,antoniuk2017sharp}. Then we
show that simplicial complexes give rise to 3-presentations.

For any group presentation $\langle S | R\rangle \cong \mathbb{Z}^{n}$, it is
known that $|R| \ge \binom{n}{2}$; this bound is sharp, by the presentation
$$\langle g_{1}, \ldots , g_{n} \,|\,
  g_{i}g_{j}g_{i}^{-1}g_{j}^{-1},\, i < j\rangle \cong \mathbb{Z}^{n}.$$
The bound $|R| \ge \binom{n}{2}$ already gives a $\Omega(n^{2 / 3})$ lower bound
for vertex numbers in \cref{asymptotic}(\ref{asymptotic-lower}). To strengthen
this bound, we use extremal results in discrete geometry (specifically, the
fractional Sylvester--Gallai results in
\cite{barak2013fractional,dvir2014improved,dvir2016sylvester}), to show that any
3-presentation $\langle S | R\rangle \cong \mathbb{Z}^{n}$ has $|S| =
\Omega(n^{3 / 2})$. This translates to a $\Omega(n^{3 / 4})$ lower bound in
\cref{asymptotic}(\ref{asymptotic-lower}).

We conjecture that the $\Omega(n^{3 / 4})$ bound in
\cref{asymptotic}(\ref{asymptotic-lower}) can be improved to $\Omega(n)$.
We also pose the following question for further research:

\begin{question}
\label{question-torsion}
  Can we prove an analogue to \cref{asymptotic} for fundamental group
  $(\mathbb{Z}_{k})^{n}$ (perhaps restricting $k$ to primes or prime powers)?
\end{question}

We will generally assume simplicial complexes $X$ are connected, since if $X$ is
disconnected, and $\pi_{1}(X, v) \cong G$, then the component $C$ of $X$
containing $v$ has fewer vertices than $X$, and $\pi_{1}(C, v) \cong G$. Under
this assumption, $\pi_{1}(X, v)$ is independent of $v$, so we will write
$\pi_{1}(X)$ instead.

% ## Upper bound
\section{Upper bound}

In this section, we show that for all $n \in \mathbb{N}$, there exists a
simplicial complex $X_{n}$ on $O(n)$ vertices with fundamental group
$\pi_{1}(X_{n}) \cong \mathbb{Z}^{n}$. We start by constructing a complex
$W_{n}$ on $n^{2} + n + 1$ vertices with fundamental group $\pi_{1}(W_{n}) \cong
\mathbb{Z}^{n}$, and then obtain $X_{n}$ by identifying vertices and edges.
%%%%%%%% CHANGE
We write $[n]$ for~$\{1,2,\dots,n\}$.
%%%%%%%% CHANGE

\begin{definition}
\label{jn}
  For $n \in \mathbb{N}$, the simplicial complex $W_{n}$ is defined as follows:
  \begin{itemize}
    \item
    The vertex set consists of a vertex $u$, vertices $v_{i, k}$ for $i \in
    [n]$, $k \in [2]$, and vertices $w_{i, j, k}$ for $i, j \in [n]$, $i < j$,
    $k \in [2]$.

    \item
    The edge set includes edges $\{u, v_{i, 1}\}, \{v_{i, 1}, v_{i, 2}\},
    \{v_{i, 2}, u\}$ for all $i \in [n]$.
    
    \item
    For each $i, j \in [n]$ with $i < j$, we include edges and triangles as in
    the following diagram. (The vertices and edges on the boundary are those
    defined above, and each planar region corresponds to a triangle.)
    
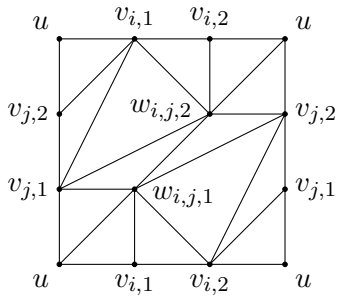
\begin{center}
    \begin{tikzpicture}
      \draw[fill] (0, 0) circle [radius = 0.03];
      \draw[fill] (0, 1) circle [radius = 0.03];
      \draw[fill] (0, 2) circle [radius = 0.03];
      \draw[fill] (0, 3) circle [radius = 0.03];
      \draw[fill] (1, 0) circle [radius = 0.03];
      \draw[fill] (1, 1) circle [radius = 0.03];
      \draw[fill] (1, 3) circle [radius = 0.03];
      \draw[fill] (2, 0) circle [radius = 0.03];
      \draw[fill] (2, 2) circle [radius = 0.03];
      \draw[fill] (2, 3) circle [radius = 0.03];
      \draw[fill] (3, 0) circle [radius = 0.03];
      \draw[fill] (3, 1) circle [radius = 0.03];
      \draw[fill] (3, 2) circle [radius = 0.03];
      \draw[fill] (3, 3) circle [radius = 0.03];

      \draw (0, 0) -- (0, 1) -- (0, 2) -- (0, 3);
      \draw (3, 0) -- (3, 1) -- (3, 2) -- (3, 3);
      \draw (0, 0) -- (1, 0) -- (2, 0) -- (3, 0);
      \draw (0, 3) -- (1, 3) -- (2, 3) -- (3, 3);
      \draw (0, 0) -- (1, 1) -- (2, 2) -- (3, 3);

      \draw (1, 0) -- (1, 1) -- (0, 1);
      \draw (2, 3) -- (2, 2) -- (3, 2);

      \draw (0, 2) -- (1, 3);
      \draw (2, 0) -- (3, 1);

      \draw (0, 1) -- (1, 3) -- (2, 2) -- (0, 1);
      \draw (1, 1) -- (2, 0) -- (3, 2) -- (1, 1);
      
      \node [below left] at (0, 0) {$u$};
      \node [above left] at (0, 3) {$u$};
      \node [below right] at (3, 0) {$u$};
      \node [above right] at (3, 3) {$u$};

      \node [left] at (0, 1) {$v_{j, 1}$};
      \node [left] at (0, 2) {$v_{j, 2}$};
      \node [right] at (3, 1) {$v_{j, 1}$};
      \node [right] at (3, 2) {$v_{j, 2}$};
      \node [below] at (1, 0) {$v_{i, 1}$};
      \node [below] at (2, 0) {$v_{i, 2}$};
      \node [above] at (1, 3) {$v_{i, 1}$};
      \node [above] at (2, 3) {$v_{i, 2}$};
      \node [right] at (1.1, 0.9) {$w_{i, j, 1}$};
      \node [left] at (1.9, 2.05) {$w_{i, j, 2}$};

    \end{tikzpicture}
    \end{center}
    (This diagram also gives a vertex-minimal triangulation of the torus.)
  \end{itemize}
\end{definition}

%%%%%%%% CHANGE Remark --> Lemma
\begin{lemma}
  For all $n \in \mathbb{N}$, we have $\pi_{1}(W_{n}) \cong \mathbb{Z}^{n}$.
\end{lemma}
%%%%%%%% CHANGE

\begin{proof}
  Note that $W_{n}$ is homeomorphic to a CW complex $W_{n}'$ consisting of:
  
  \begin{itemize}
    \item
    A single 0-cell $u$.

    \item
    A 1-cell $e_{i}$ from $u$ to itself for each $i \in [n]$, corresponding to
    the edges $\{u, v_{i, 1}\}, \{v_{i, 1}, v_{i, 2}\},$ and~$\{v_{i, 2}, u\}$.

    \item
    A 2-cell $f_{i, j}$ attached along $e_{i}e_{j}e_{i}^{-1}e_{j}^{-1}$ for each
    $i, j \in [n]$, $i < j$, corresponding to the triangles in the diagram in
    \cref{jn} for $i, j$. (By $e_{i}^{-1}$ we denote attaching in the
    opposite direction along $e_{i}$.)

  \end{itemize}
  This gives a group presentation $\pi_{1}(W_{n}') \cong \langle S | R\rangle$,
  where
  $$S = \{e_{i} : i \in [n]\},\qquad
    R = \{e_{i}e_{j}e_{i}^{-1}e_{j}^{-1} : i, j \in [n],\, i < j\}.$$
  But $\langle S | R\rangle \cong \mathbb{Z}^{n}$, so $\pi_{1}(W_{n}) \cong
  \pi_{1}(W_{n}') \cong \mathbb{Z}^{n}$, as desired.
\end{proof}

We now establish the tools we need to perform identifications on $W_{n}$.

\begin{definition}
\label{spur}
  Let $X$ be a simplicial complex with vertex $u \in V(X)$. We say that a set $S
  \subseteq V(X) \setminus \{u\}$ is a \emph{spur} in $(X, u)$ if the following
  properties hold:
  \begin{enumerate}[(1)]
    \item
    Each $v \in S$ is adjacent to $u$ in $X$.
    
    \item
    No two distinct $v, v' \in S$ are adjacent in $X$.

    \item
    No two distinct $v, v' \in S$ have a common neighbor in $X$ other than $u$.

  \end{enumerate}
\end{definition}

\begin{definition}
  Let $X$ be a simplicial complex with vertex $u \in V(X)$. We say that two
  spurs $S, S'$ in $(X, u)$ are \emph{compatible} if $S \cap S' = \emptyset$,
  and there is at most one edge $\{v, v'\}$ in $X$ with $v \in S$, $v' \in S'$.
\end{definition}

\begin{lemma}
\label{collapse}
  Let $X$ be a simplicial complex with vertex $u \in V(X)$. If $S$ is a spur in
  $(X, u)$, then we can ``collapse'' $S$ to obtain a simplicial complex, which
  we denote $X / S$, as follows:
  \begin{itemize}
    \item
    Identify all vertices $v \in S$ to a single new vertex $w$.

    \item
    Identify all edges $\{u, v\}$ for $v \in S$ to a single edge $\{u, w\}$.

  \end{itemize}
  Moreover, $\pi_{1}(X / S) \cong \pi_{1}(X)$.
\end{lemma}

\begin{proof}
  We may perform the identifications in the category of CW complexes, but we
  need to prove that the result is a simplicial complex. Since no adjacent
  vertices are identified, it remains to prove that no two distinct faces $f,
  f'$ of $X$ have the same vertex set in $X / S$, other than those explicitly
  identified.

  If $f, f'$ are distinct faces of $X$ with the same vertex set in $X / S$, then
  there exist $v, v' \in S$ with $v \in f$, $v' \in f'$. If $f, f' \subseteq
  \{u\} \cup S$, then $f, f'$ are either $\{v\}, \{v'\}$ or $\{u, v\}, \{u,
  v'\}$, and are explicitly identified. Hence we may assume that $f, f'$ both
  contain a vertex $x \not\in \{u\} \cup S$. But then $x$ is a common neighbor
  of $v, v'$, a contradiction. Hence $X / S$ is a simplicial complex.

  Now let $A$ be the subcomplex of $X$ with vertices $u, S$ and edges $\{u, v\}$
  for all $v \in S$, and let $B$ be the subcomplex of $X / S$ with vertices $u,
  w$ and edge $\{u, w\}$. Consider the quotients $X / A$, $(X / S) / B$ in the
  category of CW complexes, and note that $X / A \cong (X / S) / B$. Since $A,
  B$ are contractible, we have homotopy equivalences $X / A \simeq X$ and $(X /
  S) / B \simeq X / S$ (see e.g.,~\cite[Prop.~0.17]{hatcher2002algebraic}). By
  transitivity, we have $X \simeq X / S$, so $\pi_{1}(X) \cong \pi_{1}(X / S)$
  as desired.
\end{proof}

\begin{lemma}
\label{collapse-compatible}
  Let $X$ be a simplicial complex with vertex $u \in V(X)$.
  \begin{enumerate}[(a)]
    \item
    If $S, S'$ are compatible spurs in $(X, u)$, then $S'$ is a spur in $(X / S,
    u)$.

    \item
    If $S, S', S''$ are pairwise compatible spurs in $(X, u)$, then $S', S''$
    are compatible spurs in $(X / S, u)$.

  \end{enumerate}
\end{lemma}

\begin{proof}
  For (a), conditions (1), (2) in \cref{spur} hold since $S, S'$ are disjoint,
  and condition (3) holds since $S, S'$ have at most one edge between them.

  For (b), $S', S''$ are spurs in $(X / S, u)$ by (a), and compatibility follows
  from the fact that collapsing $S$ does not affect vertices or edges among $S'
  \cup S''$.
\end{proof}

\begin{lemma}
\label{compatible}
  For $n \in \mathbb{N}$ with $n \ne 2, 3$, the vertices $w_{i, j, k}$ of
  $W_{2n}$ (or $W_{2n - 1}$) can be partitioned into $4n - 2$ pairwise
  compatible spurs in $(W_{2n}, u)$ (or $(W_{2n - 1}, u)$).
\end{lemma}

\begin{proof}
  A \emph{1-factorization} of the complete graph on vertex set $[2n]$ is a
  partition of its edges into perfect matchings. An \emph{orthogonal pair} of
  1-factorizations is a pair of 1-factorizations, such that no two edges appear
  in the same matching in both factorizations. By~\cite{mullin1975existence}
  (see also~\cite{horton1981room,mendelsohn1985one}), such a pair $(F_{1},
  F_{2})$ exists for all $n \in \mathbb{N}$ with $n \ne 2, 3$.

  Then for each matching $M \in F_{k}$, we construct a spur $S_{M}$ in $(W_{2n},
  u)$:
  $$S_{M} = \{w_{i, j, k} : \{i, j\} \in M\}$$
  To see that $S_{M}$ is a spur, note that the neighbors of $w_{i, j, k}$ in
  $W_{2n}$ are $u$, and some vertices of the form $v_{i, k'}, v_{j, k'}, w_{i,
  j, k'}$ for $k' \in [2]$, so conditions~(1) and~(2) hold. Since $M$ is a
  matching, condition~(3) holds.

  Then the $S_{M}$ are disjoint, and the only edges between vertices in $S_{M},
  S_{M'}$ for distinct $M, M'$ are the edges $\{w_{i, j, 1}, w_{i, j, 2}\}$,
  which arise for $M \in F_{1}$, $M' \in F_{2}$ with $\{i, j\} \in M$, $\{i, j\}
  \in M'$. Then the orthogonality of $(F_{1}, F_{2})$ implies that the $S_{M}$
  are pairwise compatible, and there are $2(2n - 1) = 4n - 2$ such $S_{M}$.

  Viewing $W_{2n - 1}$ as an induced subcomplex of $W_{2n}$, the $S_{M}$ remain
  pairwise compatible spurs in $(W_{2n - 1}, u)$, upon deleting the missing
  vertices.
\end{proof}

Our promised upper bound follows:

\upper

\begin{proof}
  Starting with $W_{2n}$ or $W_{2n - 1}$, apply \cref{compatible} to obtain $4n
  - 2$ pairwise compatible spurs. Collapse these spurs, one by one, via
  \cref{collapse}, to obtain $X_{2n}$ or $X_{2n - 1}$ with $\pi_{1}(X_{2n})
  \cong \mathbb{Z}^{2n}$, $\pi_{1}(X_{2n - 1}) \cong \mathbb{Z}^{2n - 1}$; note
  that \cref{collapse-compatible} guarantees that the remaining spurs remain
  compatible. The remaining vertices are $u$, the $v_{i, k}$, and one vertex for
  each spur, so:
  \begin{itemize}
    \item
    The number of vertices in $X_{2n}$ is $1 + 4n + (4n - 2) = 8n - 1$.

    \item
    The number of vertices in $X_{2n - 1}$ is $1 + 2(2n - 1) + (4n - 2) = 8n -
    3$.

  \end{itemize}
  This completes the proof.
\end{proof}

% ## Lower bound
\section{Lower bound}

In this section, we show that a simplicial complex $X$ with fundamental group
$\pi_{1}(X) \cong \mathbb{Z}^{n}$ has $\Omega(n^{3/4})$ vertices. We begin by
relating simplicial complexes to group presentations:

\begin{definition}
  Given a group $G$, a \emph{3-presentation} of $G$ is a group presentation
  $\langle S | R\rangle \cong G$ where each relation in $R$ is one of the
  following:
  \begin{itemize}
    \item
    $\langle\rangle$ (the empty word).

    \item
    $g^{a}$, where $g \in S$ and $a \in \mathbb{Z}$.

    \item
    $g^{a}h^{b}$, where $g, h \in S$ and $a, b \in \mathbb{Z}$.

    \item
    $g^{a}h^{b}i^{c}$, where $g, h, i \in S$ and $a, b, c \in \mathbb{Z}$.

  \end{itemize}
  We say such a word $w$ is in \emph{normal form} if the generators used are all
  distinct, and $a, b, c \ne 0$. To ``write $r \in R$ in normal form'' means to
  find $w$ in normal form such that $r$ and $w$ are conjugates in~$\langle
  S\rangle$; in this case we write $r \leadsto w$.
\end{definition}

For example, we describe a 3-presentation $\langle S | R\rangle \cong
\mathbb{Z}^{n}$, derived from the presentation for $\mathbb{Z}^{n}$ given in the
introduction:
\begin{align*}
  S &= \{g_{i} : i \in [n]\} \cup \{h_{i, j} : i, j \in [n],\, i < j\}\\
  R &= \{g_{i}g_{j}h_{i, j},\, g_{j}g_{i}h_{i, j} : i, j \in [n],\, i < j\}
\end{align*}

We will use the phrase, ``Let $\phi \colon \langle S | R \rangle \cong G$ be a
3-presentation,'' to mean, ``Let $\langle S | R \rangle \cong G$ be a
3-presentation, and fix an isomorphism $\phi \colon \langle S | R\rangle \to
G$.''

%%%%%%%% CHANGE Remark --> Lemma
\begin{lemma}
\label{normal}
  Let $\langle S | R\rangle \cong G$ be a 3-presentation. Then any relation $r
  \in R$ can be written uniquely in normal form, up to the following
  conjugacies:
  \begin{itemize}
    \item
    $g^{a}h^{b}, h^{b}g^{a}$ are conjugates in $\langle S\rangle$.

    \item
    $g^{a}h^{b}i^{c}, h^{b}i^{c}g^{a}, i^{c}g^{a}h^{b}$ are conjugates in
    $\langle S \rangle$.

  \end{itemize}
\end{lemma}
%%%%%%%% CHANGE

\begin{proof}
  If $r$ is not in normal form, we can apply one of the following steps:
  \begin{itemize}
    \item
    If $r$ has a zero exponent or identical adjacent generators, rewrite $r$
    with fewer generators. (For example, $g^{0}hi$ becomes $hi$; $g^{1}g^{2}h$
    becomes $g^{3}h$.)

    \item
    If $r = g^{a}h^{b}g^{c}$, then replace $r$ with its conjugate $g^{a +
    c}h^{b}$.

  \end{itemize}
  Each such step reduces $k$ in $r = \prod_{i = 1}^{k}g_{i}^{a_{i}}$, so this
  process terminates. Uniqueness follows from considering the conjugates of
  reduced words $w$.
\end{proof}

\begin{lemma}
\label{3-presentation}
  If $X$ is a simplicial complex on $k$ vertices with fundamental group
  $\pi_{1}(X) \cong G$, then there exists a 3-presentation $\langle S | R\rangle
  \cong G$ with $|S| \le \binom{k}{2}$ and $|R| \le \binom{k}{3}$.
\end{lemma}

\begin{proof}
  Assume $X$ is connected, otherwise reduce to the component of $X$ containing
  the basepoint. Then the 1-skeleton of $X$ is a connected graph; choose a
  spanning tree $T$ of this graph. View $X$ as a CW complex and $T$ as a
  contractible subcomplex of $X$, and consider the quotient complex $X / T$,
  which is homotopy equivalent to $X$ (see e.g.,~\cite[Prop.~0.17]{hatcher2002algebraic}), so $\pi_{1}(X / T) \cong G$.

  Now $X / T$ has a single 0-cell, and hence can be viewed as the presentation
  complex of some group presentation $\langle S | R\rangle \cong G$ upon
  choosing a direction for each 1-cell. The 1-cells correspond bijectively to
  the generators, and arise from distinct edges of $X$, so $|S| \le
  \binom{k}{2}$. The 2-cells correspond bijectively to the relations, and arise
  from distinct triangles of $X$, so $|R| \le \binom{k}{3}$.

  Moreover, each $r \in R$ is of one of the following forms, depending on how
  many edges of the corresponding triangle in $X$ lie in $T$:
  \begin{itemize}
    \item
    $g^{a}$, where $g \in S$ are distinct and $a \in \{\pm 1\}$.

    \item
    $g^{a}h^{b}$, where $g, h \in S$ are distinct and $a, b \in \{\pm 1\}$.

    \item
    $g^{a}h^{b}i^{c}$, where $g, h, i \in S$ are distinct and $a, b, c \in \{\pm
    1\}$.

  \end{itemize}
  In particular, $\langle S | R\rangle$ is a 3-presentation, which completes the
  proof.
\end{proof}

Hence we may turn our attention to proving lower bounds on $|S|$ and $|R|$ for
3-presentations $\langle S | R\rangle$ of given groups. We will use the concept
of deficiency:

%%%%%%%% CHANGE finite presentations
\begin{definition}
  The \emph{deficiency} of a group presentation $P = \langle S | R\rangle$ is
  $\text{def }P = |S| - |R|$. The \emph{deficiency}, $\text{def }G$, of a group
  $G$ is the maximum of $\text{def }P$ over all finite presentations $P$ of~$G$.
\end{definition}
%%%%%%%% CHANGE

Then we have an inequality in group homology due to
Epstein~\cite{epstein1961finite}:
$$\text{def }G \le \text{rank }H_{1}(G; \mathbb{Z}) - s(H_{2}(G; \mathbb{Z})),$$
where $s(H_{2}(G; \mathbb{Z}))$ is the minimum number of generators of $H_{2}(G;
\mathbb{Z})$. (For a related inequality in terms of free resolutions,
see~\cite{swan1965minimal}.) In particular, when $G$ is~$\mathbb{Z}^{n}$, we
obtain $\text{def }\mathbb{Z}^{n} \le n - \binom{n}{2}$. Thus we have several
constraints on the size of a presentation of $\mathbb{Z}^{n}$; if $\langle S |
R\rangle \cong \mathbb{Z}^{n}$, then
\begin{itemize}
  \item
  $|S| \ge n$.

  \item
  $|R| - |S| \ge \binom{n}{2} - n$.
  
  \item
  $|R| \ge \binom{n}{2}$ (by adding the previous two inequalities).

\end{itemize}
For the presentation of $\mathbb{Z}^{n}$ described in the introduction, we have
equality in all three of these bounds. Hence $\text{def }\mathbb{Z}^{n} = n -
\binom{n}{2}$.

\cref{3-presentation} allows us to translate these bounds into lower bounds
for the number of vertices in a simplicial complex $X$ with fundamental group
$\pi_{1}(X) \cong \mathbb{Z}^{n}$. The first inequality above gives a bound of
$\Omega(n^{1/2})$, and the third gives a stronger bound of~$\Omega(n^{2/3})$. We
present the latter bound in more detail:

%%%%%%%% CHANGE Remark --> Proposition
\begin{proposition}
  A simplicial complex $X$ with fundamental group $\pi_{1}(X) \cong
  \mathbb{Z}^{n}$ has at least $\Omega(n^{2/3})$ vertices.
\end{proposition}
%%%%%%%% CHANGE

\begin{proof}
  Let $f(n)$ be the minimum number of vertices in a simplicial complex $X_{n}$
  with fundamental group $\pi_{1}(X_{n}) \cong \mathbb{Z}^{n}$. By
  \cref{3-presentation}, for each $n$ we obtain a 3-presentation $\langle
  S_{n} | R_{n}\rangle \cong \mathbb{Z}^{n}$ with $|R_{n}| \le \binom{f(n)}{3}$.
  But $|R_{n}| \ge \binom{n}{2}$, so $\binom{f(n)}{3} \ge \binom{n}{2}$, hence
  $f(n) = \Omega(n^{2/3})$.
\end{proof}

Up to now, we have considered bounds on the size of arbitrary presentations of
$\mathbb{Z}^{n}$. Now we turn to proving that for 3-presentations $\langle S |
R\rangle \cong \mathbb{Z}^{n}$, we have a stronger bound $|S| =
\Omega(n^{3/2})$. First we introduce a notion of dimension:

\begin{definition}
  Let $\phi \colon \langle S | R\rangle \cong \mathbb{Z}^{n}$ be a
  3-presentation. Then the \emph{dimension} of a subset $S' \subseteq S$,
  denoted $\dim S'$, is
  $$\dim(\text{span}\{\phi(g) : g \in S'\}),$$
  where we view each $\phi(g)$ as a vector in $\mathbb{R}^{n} \supseteq
  \mathbb{Z}^{n}$.

  For $r \in R$, let $r \leadsto \prod_{i}g_{i}^{a_{i}}$ by \cref{normal}. The
  \emph{dimension} of $r$ is the dimension of the subset $\{g_{i}\} \subseteq
  S$. (Note that the set $\{g_{i}\}$ is independent of the choice of normal
  form, so this definition is valid.)
\end{definition}

Note that for a relation $r$ with $r \leadsto \prod_{i = 1}^{k}g_{i}^{a_{i}}$,
we have $\sum_{i = 1}^{k}a_{i}\phi(g_{i}) = 0$, a linear dependence among 
the~$\phi(g_{i})$. It follows that $\dim r < k$. In particular, all relations of a
3-presentation $\langle S | R\rangle \cong \mathbb{Z}^{n}$ have dimension at
most two.

Our next goal is to show that for 3-presentations $\langle S | R\rangle \cong
\mathbb{Z}^{n}$ with $|S|$ minimal, all nonempty relations have dimension
exactly two. To do this, we use Tietze transformations
(\cite{tietze1908topologischen}; see also \cite{lyndon2001combinatorial}):

\begin{remark}[Tietze~\cite{tietze1908topologischen}]
\label{tietze}
  Consider a group presentation $\langle S | R\rangle \cong G$. Then:

  \begin{itemize}
    \item
    Let $r$ be a word in $S$ which is zero in $\langle S | R\rangle$. Then
    $\langle S | R \cup \{r\}\rangle \cong G$.

    \item
    Let $w$ be a word in $S$, and let $g$ be fresh. Then $\langle S \cup \{g\} |
    R \cup \{g^{-1}w\}\rangle \cong G$.

  \end{itemize}
  We refer to the passage from one presentation to another in either of these
  ways, in either direction, as a \emph{Tietze transformation}.
\end{remark}

We now establish several transformations of 3-presentations:

\begin{lemma}
\label{replace1}
  Let $\langle S | R\rangle \cong \mathbb{Z}^{n}$ be a 3-presentation, and
  suppose $g = 0$ in $\langle S | R\rangle$, where $g \in S$. Then we obtain a
  3-presentation $\langle S' | R'\rangle \cong \mathbb{Z}^{n}$ where:
  \begin{itemize}
    \item
    $S' = S \setminus \{g\}$.

    \item
    $R'$ is obtained from $R$ by removing $g$ wherever it appears in relations
    $r \in R$. (For example, $ghi \in R$ becomes $hi \in R'$.)

  \end{itemize}
\end{lemma}

\begin{proof}
  We apply Tietze transformations:
  \begin{itemize}
    \item
    Add the redundant relation $g$ to $R$ to obtain~$R'$.

    \item
    Remove $g$ wherever it appears in relations $r \in R'$, except in the
    relation $g \in R'$. This is valid since $g = 0$ in $\langle S | R'\rangle$
    by the relation $g \in R'$. (Each such removal is two Tietze
    transformations, adding and removing a relation.)

    \item
    Remove the generator $g$, along with the relation $g$.

  \end{itemize}
  This gives the desired 3-presentation of $\mathbb{Z}^{n}$.
\end{proof}

\begin{lemma}
\label{replace2}
  Let $\langle S | R\rangle \cong \mathbb{Z}^{n}$ be a 3-presentation, and
  suppose $g^{a}h^{b} = 0$ in $\langle S | R\rangle$, where $g, h \in S$ are
  distinct, $a, b \ne 0$, and $a, b$ are relatively prime. Then we obtain a
  3-presentation $\langle S' | R'\rangle \cong \mathbb{Z}^{n}$ where:

  \begin{itemize}
    \item
    $S' = S \cup \{i\} \setminus \{g, h\}$, where $i$ is a fresh generator.

    \item
    $R'$ is obtained from $R$ by replacing $g$ with $i^{b}$ and $h$ with
    $i^{-a}$ wherever they appear in relations $r \in R$.

  \end{itemize}
\end{lemma}

\begin{proof}
  There exist $c, d \in \mathbb{Z}$ with $ac + bd = 1$. We apply Tietze
  transformations:
  \begin{itemize}
    \item
    Add the relation $g^{a}h^{b}$, which is redundant by assumption.

    \item
    Add a generator $i$, along with the relation $i^{-1}g^{d}h^{-c}$, to obtain
    a new 3-presentation $$\phi' \colon \langle S' | R'\rangle \cong
    \mathbb{Z}^{n}.$$

    \item
    Add the relation $g^{-1}i^{b}$, which is redundant since
    $$i^{b} = g^{bd}h^{-bc} = g^{1 - ac}h^{-bc} = g(g^{a}h^{b})^{-c} = g$$
    in $\langle S' | R'\rangle$. (We use commutativity of $g, h$ in $\langle S'
    | R'\rangle$, which follows from commutativity of $\phi'(g), \phi'(h)$ in
    $\mathbb{Z}^{n}$.)

    \item
    Similarly, add the relation $h^{-1}i^{-a}$, which is redundant since
    $$i^{-a} = g^{-ad}h^{ac} = g^{-ad}h^{1 - bd} = h(g^{a}h^{b})^{-d} = h$$
    in $\langle S' | R'\rangle$.

    \item
    Replace $g$ with $i^{b}$ and $h$ with $i^{-a}$ wherever they appear in
    relations $r \in R'$ (i.e. in all relations other than the new relations
    $g^{-1}i^{b}, h^{-1}i^{-a}$).

    \item
    Remove the generators $g, h$, along with the relations $g^{-1}i^{b},
    h^{-1}i^{-a}$.

    \item
    Remove the relation $i^{-1}g^{d}h^{-c}$, which is now $i^{-1}i^{bd}i^{ac} =
    0$.

    \item
    Remove the relation $g^{a}h^{b}$, which is now $(i^{b})^{a}(i^{-a})^{b} =
    0$.

  \end{itemize}
  This gives the desired 3-presentation of $\mathbb{Z}^{n}$.
\end{proof}

\begin{lemma}
\label{minimal}
  Let $\phi \colon \langle S | R\rangle \cong \mathbb{Z}^{n}$ be a 3-presentation
  with $|S|$ minimal. Then for each nonempty $r \in R$, we have $r \leadsto
  g^{a}h^{b}i^{c}$, where $g, h, i \in S$ are distinct and $a, b, c\ne 0$, and
  $\dim r = 2$.
\end{lemma}

\begin{proof}
  Let $r \leadsto \prod_{i = 1}^{k}g_{i}^{a_{i}}$ by \cref{normal}. By
  \cref{replace1}, no $g_{i}$ are zero in $\langle S | R\rangle$, which implies
  $k \ne 1$. By \cref{replace2}, no distinct $g_{i}, g_{j}$ have $\phi(g_{i})$,
  $\phi(g_{j})$ in a common one-dimensional subspace of $\mathbb{R}^{n}$, which
  implies $k \ne 2$.

  Hence $k = 3$, so $r \leadsto g^{a}h^{b}i^{c}$ for $g, h, i \in S$ distinct
  and $a, b, c \ne 0$. Then the considerations above imply $\dim\{g, h\} = 2$,
  so $\dim r \ge 2$. Since $\phi(g), \phi(h), \phi(i)$ are dependent in
  $\mathbb{R}^{n}$, we have $\dim r = 2$.
\end{proof}

Our next transformation requires the notion of a \emph{sparse} set of relations:

\begin{definition}
  Let $\langle S | R\rangle \cong \mathbb{Z}^{n}$ be a 3-presentation, and let
  $S' \subseteq S$, $R' \subseteq R$. Then define the set $R'[S'] \subseteq R'$
  as
  $$R'[S'] = \{r \in R' : r \leadsto w,\,
    \text{and $w$ uses only generators in $S'$}\}.$$
\end{definition}

\begin{definition}
  Let $\phi \colon \langle S | R\rangle \cong \mathbb{Z}^{n}$ be a
  3-presentation, and let $R' \subseteq R$.
  \begin{itemize}
    \item
    For $S' \subseteq S$ with $\dim S' = 2$, $R'$ is \emph{sparse on $S'$} if
    $|R'[S']| \le |S'| - 1$.

    \item
    $R'$ is \emph{sparse} if $R'$ is sparse on all $S' \subseteq S$ with $\dim
    S' = 2$.

    \item
    A set $S' \subseteq S$ is \emph{critical for $R'$} if $\dim S' = 2$ and
    $|R'[S']| = |S'| - 1$.
  \end{itemize}
\end{definition}

%%%%%%%% CHANGE Remark --> Lemma
\begin{lemma}
\label{critical-union}
  Let $\phi \colon \langle S | R\rangle \cong \mathbb{Z}^{n}$ be a
  3-presentation with $|S|$ minimal, so that \cref{minimal} applies. Suppose $R'
  \subseteq R$ is sparse, and $S', S'' \subseteq S$ are critical for $R'$. If
  $R[S'] \cap R[S''] \ne \emptyset$, then $S' \cup S''$ is also critical for
  $R'$.
\end{lemma}
%%%%%%%% CHANGE

\begin{proof}
  Let $r \in R[S'] \cap R[S'']$, and write $r \leadsto g^{a}h^{b}i^{c}$ by
  \cref{minimal}. Then the set $\{\phi(g), \phi(h), \phi(i)\}$ spans a
  2-dimensional subspace $U \subseteq \mathbb{R}^{n}$. Since $\dim S' = \dim S''
  = 2$, we have $\text{span}(\phi(S')) = \text{span}(\phi(S'')) = U$. Then
  $$U \subseteq \text{span}(\phi(S' \cap S'')) \subseteq \text{span}(\phi(S')) =
  U,$$
  so $\text{span}(\phi(S' \cap S'')) = U$, and also $\text{span}(\phi(S' \cup
  S'')) = U + U = U$. In particular, $\dim(S' \cap S'') = \dim(S' \cup S'') =
  2$. Therefore, we have
  \begin{align*}
    |R'[S' \cup S'']|
      &= |R'[S']| + |R'[S'']| - |R'[S' \cap S'']|\\
      &\ge (|S'| - 1) + (|S''| - 1) - (|S' \cap S''| - 1)\\
      &\ge |S' \cup S''| - 1.
  \end{align*}
  Hence $S' \cup S''$ is critical for $R'$.
\end{proof}

\begin{corollary}
\label{critical}
  Let $\phi \colon \langle S | R \rangle \cong \mathbb{Z}^{n}$ be a
  3-presentation with $|S|$ minimal, so that \cref{minimal} applies. Suppose $R'
  \subseteq R$ is sparse. Then there exists a collection $\mathcal{C}$ of
  certain critical sets $S' \subseteq S$ for $R'$, such that:
  \begin{enumerate}[(1)]
    \item
    If $S'' \subseteq S$ is critical for $R'$, then there exists $S' \in
    \mathcal{C}$ with $S'' \subseteq S'$.

    \item
    If $S', S'' \in \mathcal{C}$, then $R[S'] \cap R[S''] = \emptyset$.

  \end{enumerate}
\end{corollary}

\begin{proof}
  First take the collection $\mathcal{C} = \{S' \subseteq S : S' \text{ critical for $R'$}\}$;
  then (1) holds. If $S', S'' \in \mathcal{C}$ with ${R[S'] \cap R[S''] \ne \emptyset}$, 
  then by~\cref{critical-union}, $S' \cup S''$ is critical for~$R'$.
  Then consider removing $S', S''$ from $\mathcal{C}$, and adding $S' \cup S''$
  if it is not present.

  While (2) fails, apply the step above repeatedly. Each step preserves (1) and
  reduces $|\mathcal{C}|$, so this process terminates with $\mathcal{C}$ such
  that (1), (2) both hold.
\end{proof}

\begin{lemma}
\label{replace-sparse}
  Let $\phi \colon \langle S | R\rangle \cong \mathbb{Z}^{n}$ be a
  3-presentation with $|S|$ minimal, so that \cref{minimal} applies. Partition
  $R$ as $R = R_{s} \sqcup R_{e} \sqcup R_{o}$ (mnemonic: ``sparse,'' ``extra,''
  ``other''), such that $R_{s}$ is sparse, and for each $r \in R_{e}$ with $r
  \leadsto g^{a}h^{b}i^{c}$, we have $\{g, h, i\} \subseteq S'$ for some
  critical $S' \subseteq S$ for $R_{s}$. Then we obtain a 3-presentation
  $\langle S' | R'\rangle \cong \mathbb{Z}^{n}$ where:
  \begin{itemize}
    \item
    $S'$ includes all generators in $S$.

    \item
    $R'$ includes all relations in $R_{o}$.

    \item
    $|R'| - |S'| = |R_{s}| + |R_{o}| - |S|$.

  \end{itemize}
\end{lemma}

\begin{proof}
  Obtain a collection $\mathcal{C}$ of critical sets for $R_{s}$ via
  \cref{critical}. Then for each $S'' \in \mathcal{C}$, consider the
  integer span of $\phi(S'')$ in $\mathbb{Z}^{n}$, that is, the set
  $$\Lambda = \left\{\sum_{i = 1}^{k}a_{i}\phi(g_{i})
    : k \in \mathbb{N},\, a_{i} \in \mathbb{Z},\, g_{i} \in S''\right\}.$$
  We have $\Lambda \cong \mathbb{Z}^{2}$ (see~\cite[Thm~1.12.3]{duistermaat2000lie}), 
  so $\Lambda$ has a basis $\{x_{1}, x_{2}\}$. We apply Tietze
  transformations (for each $S'' \in \mathcal{C}$) to $\langle S | R\rangle$.
  (We will introduce some relations with more than three generators, but we
  remove these later.)
  \begin{itemize}
    \item
    For each $j \in [2]$, write $x_{j} = \sum_{i}a_{i}\phi(g_{i})$ for $a_{i}
    \in \mathbb{Z}$, $g_{i} \in S''$. Then add a generator $h_{j}$, along with
    the relation $h_{j}^{-1}\prod_{i}g_{i}^{a_{i}}$, to obtain $\phi' \colon
    \langle S' | R'\rangle \cong \mathbb{Z}^{n}$. Note that
    $$\phi'(h_{j})
      = \phi\left(\prod_{i}g_{i}^{a_{i}}\right)
      = \sum_{i}a_{i}\phi(g_{i})
      = x_{j}.$$

    \item
    For each $g \in S''$, write $\phi(g) = \sum_{j}b_{j}x_{j}$ for $b_{j} \in
    \mathbb{Z}$. Then add the relation $g^{-1}\prod_{j}h_{j}^{b_{j}}$, which is
    redundant since
    $$\phi'\left(g^{-1}\prod_{j}h_{j}^{b_{j}}\right)
      = -\phi(g) + \sum_{j}b_{j}\phi'(h_{j})
      = 0,$$
    where we use $\phi'(h_{j}) = x_{j}$ in the last step.

    \item
    Add a generator $h_{*}$, along with the relation $h_{*}^{-1}h_{1}h_{2}$.

    \item
    Add the relation $h_{*}^{-1}h_{2}h_{1}$, which is redundant since
    $\mathbb{Z}^{n}$ (and hence our current $\langle S' | R'\rangle \cong
    \mathbb{Z}^{n}$) is abelian. Note that the relation
    $h_{1}h_{2}h_{1}^{-1}h_{2}^{-1}$ is now implied by the relations
    $h_{*}^{-1}h_{1}h_{2}$ and~$h_{*}^{-1}h_{2}h_{1}$.

    \item
    Remove all relations $r \in R[S'']$, which are now redundant. To see this,
    first rewrite $r$ in terms of only the $h_{j}$, via the relations
    $g^{-1}\prod_{j}h_{j}^{b_{j}}$. Then rewrite $r$ as $\prod_{j}h_{j}^{b_{j}}$
    for $b_{j} \in \mathbb{Z}$, via the relations
    $h_{i}h_{j}h_{i}^{-1}h_{j}^{-1}$. Applying $\phi'$, we obtain
    $\sum_{j}b_{j}x_{j} = 0$, so $b_{j} = 0$ by the lattice structure of
    $\Lambda \cong \mathbb{Z}^{2}$. Hence we have rewritten $r$ as the empty
    word, so $r$ is redundant.

    \item
    Remove the relations $h_{j}^{-1}\prod_{i}g_{i}^{b_{i}}$ added in the first
    step, which are now redundant, since we may rewrite any such relation in
    terms of only the $h_{j}$, and then apply the previous argument.

  \end{itemize}
  After applying these steps for each $S'' \in \mathcal{C}$, we call the
  resulting 3-presentation $\langle S' | R'\rangle$. For each $S'' \in
  \mathcal{C}$, we have added three generators and a net of $|S''| - |R[S'']| +
  2$ relations. By definition of~$\mathcal{C}$, the sets $R[S'']$ are disjoint
  for distinct $S'' \in \mathcal{C}$. Hence we have
  \begin{align*}
    |R'| - |S'|
      &= |R| - |S| + \sum_{S'' \in \mathcal{C}}(|S''| - |R[S'']| - 1)\\
      &= |R| - |S| + \sum_{S'' \in \mathcal{C}}(|S''| - |R_{s}[S'']| - 1)
        - \sum_{S'' \in \mathcal{C}}|R_{e}[S'']|\\
      &= |R| - |S| - |R_{e}|\\
      &= |R_{s}| + |R_{o}| - |S|.
  \end{align*}
  This completes the proof.
\end{proof}

We need one more transformation of presentations:

\begin{lemma}
\label{replace-subspace}
  Let $\phi \colon \langle S | R\rangle \cong \mathbb{Z}^{n}$ be a
  3-presentation, let $S' \subseteq S$, and let $d = \dim S'$. Then we obtain a
  presentation $\langle S'' | R''\rangle \cong \mathbb{Z}^{n - d}$ where:
  \begin{itemize}
    \item
    $S'' = S \setminus S'$.

    \item
    $R''$ is obtained from $R$ by adding $d$ relations to form $R'$, then
    removing each $g \in S'$ wherever it appears in relations $r \in R'$.

  \end{itemize}
\end{lemma}

\begin{proof}
  Let $U = \text{span}(\phi(S'))$ in $\mathbb{R}^{n}$. Then $U \cap
  \mathbb{Z}^{n}$ is a lattice of dimension $d$, so we may take a basis
  $\{x_{1}, \ldots , x_{d}\}$ of $U \cap \mathbb{Z}^{n}$, and extend to a basis
  $\{x_{1}, \ldots , x_{n}\}$ of $\mathbb{Z}^{n}$ (see Chapter~2, Lemma~4
  of~\cite{nguyen2010lll}). Then for each $i \in [d]$, let $w_{i}$ be a word
  in $\langle S\rangle$ with $\phi(w_{i}) = x_{i}$ in $\mathbb{Z}^{n}$. Let $R'
  = R \cup \{w_{1}, \ldots , w_{d}\}$.

  We claim $\langle S | R'\rangle \cong \mathbb{Z}^{n - d}$. To prove this, we
  will construct an isomorphism $\psi \colon \langle S | R'\rangle \to
  \mathbb{Z}^{n - d}$. Let $p \colon \mathbb{Z}^{n} \to \mathbb{Z}^{n - d}$ be
  the projection to the last $n - d$ coordinates under the basis $\{x_{1},
  \ldots , x_{n}\}$; more precisely,
  $$p\left(\sum_{i = 1}^{n}a_{i}x_{i}\right) = \sum_{i = d + 1}^{n}a_{i}y_{i},$$
  where $\{y_{d + 1}, \ldots , y_{n}\}$ is a basis for $\mathbb{Z}^{n - d}$.
  Note that $p$ is linear. Now define $\psi$ on $S$ by $\psi(g) = p(\phi(g))$
  for all $g \in S$, and extend $\psi$ to $\langle S\rangle$ by the universal
  property of the free group. Then for any word $w = \prod_{i}g_{i}^{a_{i}}$ in
  $\langle S\rangle$, we have
  $$\psi(w)
    = \sum_{i}a_{i}\psi(g_{i})
    = \sum_{i}a_{i}p(\phi(g_{i}))
    = p\left(\sum_{i}a_{i}\phi(g_{i})\right)
    = p(\phi(w)).$$
  In particular, for $r \in R$, we have $\psi(r) = p(\phi(r)) = p(0) = 0$. For
  the $w_{i}$ above, we have $\psi(w_{i}) = p(\phi(w_{i})) = p(x_{i}) = 0$.
  Therefore, $\psi$ is well-defined on $\langle S | R'\rangle$.

  To show $\psi$ is injective, suppose $\psi(w) = 0$ for $w \in \langle
  S\rangle$. Then $p(\phi(w)) = 0$, so $\phi(w) = \sum_{i = 1}^{d}a_{i}x_{i}$
  for some $a_{i} \in \mathbb{Z}$. Then $\phi(w) = \phi(\prod_{i =
  1}^{d}w_{i}^{a_{i}})$ so $w = \prod_{i = 1}^{d}w_{i}^{a_{i}}$ in $\langle S |
  R\rangle$ by the injectivity of $\phi$. Since $R \subseteq R'$, we have $w =
  \prod_{i = 1}^{d}w_{i}^{a_{i}}$ in $\langle S | R'\rangle$ also. But since
  $w_{i} \in R$, this implies $w = 0$ in $\langle S | R'\rangle$.

  To show $\psi$ is surjective, it suffices to show that for each $d < i \le n$,
  there exists $w \in \langle S\rangle$ with $\psi(w) = y_{i}$. By the
  surjectivity of $\phi$, take $w$ with $\phi(w) = x_{i}$. Then $\psi(w) =
  p(\phi(w)) = p(x_{i}) = y_{i}$ as desired. Hence $\langle S | R'\rangle \cong
  \mathbb{Z}^{n - d}$.

  Now all generators $g \in S'$ have $g = 0$ in $\langle S | R'\rangle$, so
  repeated application of \cref{replace1} gives the desired result.
\end{proof}

Next, we need a variant of the Sylvester--Gallai-type results
in~\cite{barak2013fractional,dvir2014improved,dvir2016sylvester}. We begin by
stating the relevant definition and theorem from~\cite{dvir2014improved}:

\begin{definition}[\cite{dvir2014improved}, Definition 1.7]
  Given a set of points $v_{1}, \ldots , v_{n} \in \mathbb{R}^{d}$, a
  \emph{special line} is a line in $\mathbb{R}^{d}$ containing at least three of
  the points $v_{i}$. We say that $v_{1}, \ldots , v_{n}$ is a \emph{$\delta$-SG
  configuration} if for each $v_{i}$, $i \in [n]$, at least $\delta (n - 1)$ of
  the remaining points lie on special lines through $v_{i}$.
\end{definition}

\begin{theorem}[\cite{dvir2014improved}, Theorem 5.1]
\label{configuration-dim}
  If $v_{1}, \ldots , v_{n}$ is a $\delta$-SG configuration, then the affine
  dimension of $\{v_{1}, \ldots , v_{n}\}$ is at most $12 / \delta$.
\end{theorem}

Now we give our variant; we translate the average case result in
\cite{dvir2014improved} from an affine setting to a linear one (as
in~\cite{dvir2016sylvester}), with a guarantee on $|E'|$:

\begin{theorem}
\label{sylvester}
  Let $V \subseteq \mathbb{R}^{d} \setminus \{0\}$ be a finite set of points,
  such that no two points in $V$ lie in a common 1-dimensional subspace of
  $\mathbb{R}^{d}$. Let $E$ be a finite set of (not necessarily distinct)
  triples $\{u, v, w\}$ of distinct points $u, v, w \in V$ lying in a common
  2-dimensional subspace of $\mathbb{R}^{d}$, so that $(V, E)$ forms a 3-uniform
  hypergraph. Suppose that for each induced subhypergraph $(V', E')$ of $(V, E)$
  with $\dim(\emph{\text{span }}V') \le 2$, we have $|E'| \le |V'| - 1$. Then
  for $\lambda > 0$, there exists an induced subhypergraph $(V', E')$ of $(V,
  E)$ with $|E| - |E'| < \lambda |V|$, and
  $$\dim(\emph{\text{span }}V') \le 12|V| / \lambda.$$
\end{theorem}

\begin{proof}
  Following the proof of~\cite[Thm.~13]{barak2013fractional}, consider $(V,
  E)$ as a 3-uniform hypergraph, and repeatedly remove vertices of degree less
  than $\lambda$. This removes less than $\lambda |V|$ edges, so we obtain a
  sub-hypergraph $(V', E')$ with $|E| - |E'| < \lambda |V|$ and minimum degree
  at least $\lambda$.

  Fix $u \in V'$; the neighborhood $N(u)$ in $(V', E')$ forms a graph $G(u)$,
  where we consider two vertices $v, w \in N(u)$ adjacent if and only if $\{u, v, w\} \in
  E'$. If $v, w \in N(u)$ are adjacent in $G(u)$, then $w$ lies in
  $\text{span}(\{u, v\}) \subseteq \mathbb{R}^{d}$. Therefore, if $\{v_{1},
  \ldots , v_{k}\}$ form a component $C$ of $G(u)$, then $U = \{u, v_{1}, \ldots
  , v_{k}\}$ has $\dim(\text{span }U) \le 2$, so the number of triples in $E'$
  using only points in $U$ is at most $k$. Hence the number of edges in $C$ is
  at most $k$. Summing over components $C$, the number of neighbors of $v$ in
  $(V', E')$ is at least $\deg_{(V', E')}v \ge \lambda$.

  Now choose a nonzero vector $\vec{n} \in \mathbb{R}^{d}$ not orthogonal to any
  $v \in V'$, and define an affine hyperplane $H = \{\vec{x} \in \mathbb{R}^{d}
  : \vec{x} \cdot \vec{n} = 1\}$. Then to each $v \in V'$ we associate the
  unique point $\tilde{v} \in \text{span}(\{v\}) \cap H$. Note that the
  $\tilde{v}$ are distinct, since no two points in $V$ lie in a common
  1-dimensional subspace of $\mathbb{R}^{d}$. Also, note that $u, v, w$ lie in a
  common two-dimensional subspace of $\mathbb{R}^{d}$ if and only if $\tilde{u},
  \tilde{v}, \tilde{w}$ lie on a common line in~$H$. Then the set $\tilde{V'} =
  \{\tilde{v} : v \in V'\}$ is a $\delta$-SG configuration with $\delta =
  \lambda / |V|$. By \cref{configuration-dim}, the affine dimension of
  $\tilde{V'}$ is at most $12 / \delta$, so $\dim(\text{span }V') \le 12|V| /
  \lambda$ as desired.
\end{proof}

Now we prove our bound on the size of 3-presentations of $\mathbb{Z}^{n}$:

\begin{theorem}
  If $\langle S | R\rangle \cong \mathbb{Z}^{n}$ is a 3-presentation, then $|S|
  = \Omega(n^{3/2})$.
\end{theorem}

\begin{proof}
  Fix an isomorphism $\phi \colon \langle S | R\rangle \to \mathbb{Z}^{n}$.
  Assume that $|S|$ is minimal, and consider the images $\phi(g)$ for $g \in S$.
  By \cref{replace1}, all $\phi(g)$ are nonzero; by \cref{replace2}, no two
  $\phi(g)$ lie in a common 1-dimensional subspace of $\mathbb{R}^{d}$.
  Moreover, by \cref{minimal}, for each $r \in R$ we have $r \leadsto
  g^{a}h^{b}i^{c}$ for $g, h, i \in S$ distinct, and $\dim r = 2$. Let $R'$ be
  an inclusion-wise maximal sparse subset of $R$.

  Now let $k = |S|$, let $c > 0$ be a constant to be determined later, and apply
  \cref{sylvester} with $V = \phi(S)$, $E = \{\{\phi(g), \phi(h), \phi(i)\} : r
  \leadsto g^{a}h^{b}i^{c},\, r \in R'\}$, and $\lambda = ck / n$, to obtain $S'
  \subseteq S$ such that:
  \begin{enumerate}[(1)]
    \item
    $|R' \setminus R[S']| \le ck^{2} / n$.

    \item
    $\dim S' \le 12|V| / \lambda = 12n / c$.

  \end{enumerate}
  If there exists $g \in S \setminus S'$ with $\dim (S' \cup \{g\}) = \dim S'$,
  then we may replace $S'$ with $S' \cup \{g\}$, preserving (1) and (2).
  Therefore, we may assume that for each $g \in S \setminus S'$, we have
  $\phi(g) \not\in \text{span }\phi(S')$.

  Now partition $R$ as $R = R_{s} \sqcup R_{e} \sqcup R_{o}$, where:
  \begin{itemize}
    \item
    $R_{s} = R' \setminus R[S']$.

    \item
    $R_{e} = (R \setminus R') \setminus R[S']$.

    \item
    $R_{o} = R[S']$.
  
  \end{itemize}
  Note that $|R_{s}| \le ck^{2} / n$ by the above. The set $R_{s}$ is sparse, since
  sparseness is closed under taking subsets. Also, for each $r \in R_{e}$ with
  $r \leadsto g^{a}h^{b}i^{c}$, we have $\{g, h, i\} \subseteq S''$ for some
  critical $S'' \subseteq S$ for~$R'$, since $r \not\in R'$ and $R'$ is maximal.
  But since $r \not\in R[S']$, we have $\{g, h, i\} \not\subseteq S'$; assume
  WLOG $g \not\in S'$. Then $\phi(g) \not\in \text{span }\phi(S')$ by the above,
  so $\text{span }\phi(S'') \not\subseteq \text{span }\phi(S')$. Then $R[S'']
  \cap R[S'] = \emptyset$, since any $r' \in R[S'']$ determines the
  2-dimensional subspace $\text{span }\phi(S'')$. Therefore, $S''$ is also
  critical for $R_{s}$. Hence we may apply \cref{replace-sparse} to $\langle S |
  R\rangle$, to obtain a 3-presentation $\langle S'' | R''\rangle \cong
  \mathbb{Z}^{n}$ with $|R''| - |S''| = |R_{s}| + |R_{o}| - |S|$.

  Finally, let $d = \dim S'$, and apply \cref{replace-subspace} to $\langle S''
  | R''\rangle$ using $S' \subseteq S''$, to obtain $\langle S''' | R'''\rangle
  \cong \mathbb{Z}^{n - d}$. Then remove all relations in $R'''$ arising from
  relations in $R_{o} = R[S']$, which are now trivial, to obtain $\langle S''' |
  R''''\rangle \cong \mathbb{Z}^{n - d}$. Then
  \begin{align*}
  |R''''| - |S'''|
    &= (|R'''| - |R_{o}|) - |S'''|\\
    &= (|R''| + d - |R_{o}|) - (|S''| - |S'|)\\
    &= (|R''| - |S''| - |R_{o}|) + d + |S'|\\
    &= (|R_{s}| - |S|) + d + |S'|\\
    &= |R_{s}| + d - |S \setminus S'|\\
    &\le ck^{2} / n + d.
  \end{align*}
  But by the bound $\text{def }\mathbb{Z}^{m} \le m - \binom{m}{2}$, we have
  $|R''''| - |S'''| = \Omega((n - d)^{2})$. Take $c = 24$; then $d \le 12n / c =
  n / 2$, so $n - d \ge n / 2$. Hence $|R''''| - |S'''| = \Omega(n^{2})$. Since
  $d \le n / 2$, we have $ck^{2} / n = \Omega(n^{2})$. Therefore, $k =
  \Omega(n^{3 / 2})$ as desired.
\end{proof}

\begin{theorem}
\label{lower}
  A simplicial complex $X$ with fundamental group $\pi_{1}(X) \cong
  \mathbb{Z}^{n}$ has at least $\Omega(n^{3/4})$ vertices.
\end{theorem}

\begin{proof}
  Let $f(n)$ be the minimum number of vertices in a simplicial complex $X_{n}$
  with fundamental group $\pi_{1}(X_{n}) \cong \mathbb{Z}^{n}$. By
  \cref{3-presentation}, for each $n$ we obtain a 3-presentation $\langle
  S_{n} | R_{n}\rangle \cong \mathbb{Z}^{n}$ with $|S_{n}| \le \binom{f(n)}{2}$.
  But $|S_{n}| = \Omega(n^{3/2})$, so $\binom{f(n)}{2} = \Omega(n^{3/2})$, hence
  $f(n) = \Omega(n^{3/4})$.
\end{proof}

% ## Acknowledgements

\section*{Acknowledgements}

We thank Wesley Pegden for pointing out the connection to 1-factorizations, and
Boris Bukh for pointing out the connection to the Sylvester--Gallai results.

% ## Backmatter

\linespread{1.13}

\bibliographystyle{plain}

\end{document}